\documentclass{article}

\usepackage{PRIMEarxiv}

\usepackage[utf8]{inputenc} % allow utf-8 input
\usepackage[T1]{fontenc}    % use 8-bit T1 fonts
\usepackage{hyperref}       % hyperlinks
\usepackage{url}            % simple URL typesetting
\usepackage{booktabs}       % professional-quality tables
\usepackage{amsfonts}       % blackboard math symbols
\usepackage{nicefrac}       % compact symbols for 1/2, etc.
\usepackage{microtype}      % microtypography
\usepackage{lipsum}
\usepackage{fancyhdr}       % header
\usepackage{graphicx}       % graphics
\graphicspath{{media/}}     % organize your images and other figures under media/ folder
\usepackage{amsmath}
\usepackage{amssymb}
\usepackage{amsthm}

\theoremstyle{plain}
\newtheorem{theorem}{Theorem}[section]

\theoremstyle{definition}

\newtheorem{definition}{Definition}

\title{Advances in the characterization of curvature of probability manifolds generated by two-dimensional families of distributions.
}

\author{
	Giuseppe Giacopelli \\
	CNR IRIB, Palermo, Italy \\
	\texttt{giuseppeg94@gmail.com} \\
	\And
	Andrea De Gaetano \\
	CNR IRIB, Palermo, Italy \\
	CNR IASI, Rome, Italy \\
	Óbuda University, Budapest, Hungary \\
	\texttt{andrea.degaetano@cnr.it} \\
}

\begin{document}
	\maketitle

	\begin{abstract}
		In this work some advances in the theory of curvature of two-dimensional probability manifolds corresponding to families of distributions are proposed. It is proved that location-scale distributions are hyperbolic in the Information Geometry sense even when the generatrix is non-even or non-smooth. A novel formula is obtained for the computation of curvature in the case of exponential families: this formula  implies some new flatness criteria in dimension 2. Finally, it is observed that many two parameter distributions, widely used in applications, are locally hyperbolic, which highlights the role of hyperbolic geometry in the study of commonly employed probability manifolds. These results have benefited from the use of explainable computational tools, which can substantially boost scientific productivity.
	\end{abstract}

	\maketitle

	%------
	% INSERT THE BODY OF THE PAPER HERE (except
	% acknowledgments, funding info and bibliography)
	%------

	\section{Introduction}
	\subsection{State of the Art}
	Over the past few years the theme of ``Explainable" AI (henceforth XAI) has become a pervasive theme in AI discussions, both in Europe (with the AI Act) and also, more recently, in the US. What XAI actually represents is currently debated: in the present work we use the operational definition that an XAI is an Artificial Intelligence algorithm whose results can be easily checked by humans (hence explained by them). In the following we will use two XAI systems to help us establish some results in the theory of probability manifolds. The tools used are a CAS (Computer Algebra System) implemented in sympy, which performs symbolic calculus, and an automatic integrator (Wolfram Alpha), which is able to progressively solve even very complicated integrals, by making recourse to a set of established integration rules. Helped by  this technology, in the present work we will propose three advances in the theory of probability manifolds in dimension 2. The first one is an extension of the results obtained in \cite{Nielsen2021} for non-even, non-smooth and arbitrary-support generatrix function $p(x)$. The second is a novel formula for the computation of curvature in the case of exponential families \cite{Amari2016}, formula which implies new  criteria for flatness in dimension 2. The third, minor, advance will be the restatement of the theorem by Le Brigant and Puechmorel \cite{Brigant2019} in a more compact way. \\
	The results obtained suggest that most experimentally employed probability distributions with two parameters are in fact locally hyperbolic, making hyperbolic geometry the main tool in the study of commonly used probability manifolds.

	\subsection{Fisher Information Matrix and Probability family manifold}
	Given a parametric  family of distributions $p(x|\theta)$, with parameter $\theta= (\theta_1, ..., \theta_d)$, the Fisher Information matrix is defined as
	$$
	I_F(\theta) = \mathbb{E}_{p(x|\theta)} [ - \nabla\nabla \log p(x|\theta) ].
	$$

	The Fisher information matrix induces the metric tensor $g_{ij}$ as
	$$
	g_{ij}(\theta) = \mathbb{E}_{p(x|\theta)} [- \partial_i \partial_j \log p(x|\theta) ]
	$$

	in the Riemann manifold of the corresponding family of probability distributions. Our aim will be to study such Riemannian manifold by studying $g_{ij}$. \\

	Supposing that the vector $\theta \in \mathbb{R}^d$, we define first of all the Christoffel symbols (in Einstein notation) as
	$$\Gamma^i_{jk} = \frac{1}{2} g^{il} \Bigg( \frac{\partial g_{lj}}{\partial x^k}+\frac{\partial g_{lk}}{\partial x^j}-\frac{\partial g_{jk}}{\partial x^l}\Bigg), $$
	the derivative of Christoffel symbols as
	$$\Gamma^{i}_{jk,\lambda} = \frac{\partial \Gamma^i_{jk}}{\partial x^{\lambda}}, $$
	the Ricci curvature as the number (dependent on $\theta$)
	$$S_R = g^{\mu \nu} \Big( \Gamma^{\lambda}_{\mu \nu, \lambda} - \Gamma^{\lambda}_{\mu \lambda, \nu} + \Gamma^{\sigma}_{\mu \nu} \Gamma^{\lambda}_{\lambda \sigma} - \Gamma^{\sigma}_{\mu \lambda} \Gamma^{\lambda}_{\nu \sigma}\Big), $$
	and the scalar curvature (henceforth simply curvature)
	$$S = \frac{1}{d} S_R. $$

	The value of the curvature is of interest because it can be proved that
	\begin{itemize}
		\item If $S>0$ the geometry of the manifold is locally spherical
		\item If $S=0$ the geometry of the manifold is locally flat
		\item If $S<0$ the geometry of the manifold is locally hyperbolic
	\end{itemize}
	In the present work the dimension will be supposed to be $d=2$.

	\subsection{Example of curvature computation}
	To make an example, suppose the parametric distributions family of Gaussians
	$$p(x|\mu,\sigma^2) = \frac{1}{\sqrt{2 \pi \sigma^2}} e^{-\frac{(x-\mu)^2}{2 \sigma^2}} $$
	is endowed with the metric tensor
	$$g_{ij}(\mu,\sigma^2) = I_F(\mu,\sigma^2) = \begin{pmatrix}
		\frac{1}{\sigma^2} & 0 \\
		0 & \frac{1}{2 \sigma^4}
	\end{pmatrix}. $$
	In this case the curvature can be computed as
	$$S(\mu,\sigma^2) = - \frac{1}{2} < 0, $$
	which implies that the curvature is negative, i.e. that the geometry is locally hyperbolic.

	\section{Location-scale families}
	\subsection{Non-smooth location-scale families with arbitrary support}
	In the present section we will follow mainly Nielsen \cite{Nielsen2021}, with some generalizations. We consider a random variable $X$ with density function $p_X(x) = p(x)$ and cumulative distribution function

	$$
	F_X(\bar{x}) = P(X \leq \bar{x}) = \int_{-\infty}^{\bar{x}} p_X(x) dx \,\,,
	$$

	where the integral is the Lebesgue integral. In \cite{Nielsen2021}  $p_X(x)$  is required to be a smooth function, with $p_X(x)>0, \forall x \in \mathbb{R}$. In the present work the requirements will be less restrictive, but some definitions are needed:

	\begin{definition}
		Given a real-valued function $p(x)$, the \emph{positive support} of the function $p(x)$ is the set
		$$
		p(x)^+ = \big\{ x \in \mathbb{R} \mid  p(x)>0\big\} \,\,.
		$$
		Similarly, the \emph{negative support} of the function $p(x)$ is the set
		$$
		p(x)^- = \big\{ x \in \mathbb{R} \mid p(x)<0\big\}.
		$$
	\end{definition}
	If $p_X(x)$ a probability distribution then $p_X(x) \geq 0$ and we will henceforth ignore the negative support since $p_X(x)^- = \emptyset$. In Nielsen \cite{Nielsen2021} it is supposed that $p_X(x)^+ = (-\infty,+\infty) = \mathbb{R}$, here we relax this constraint by introducing the concept of \emph{arbitrary support function}:

	\begin{definition}
		An \emph{arbitrary support function} p(x) is a function for which
		$$p(x)^+ \subseteq \mathbb{R}$$
		in the sense that $p(x)^+$ could be a proper subset of $\mathbb{R}$.
	\end{definition}

	In conclusion
	\begin{definition}
		A real-valued function $p_X(x)$ is a \emph{non-smooth arbitrary support distribution} if
		\begin{enumerate}
			\item The function $p_X(x)$ is non- negative ($p_X(x) \geq 0$) and has arbitrary support.
			\item The support $p_X(x)^+$ is a measurable set with non-zero measure and $p_X(x)$ is Lebesgue-integrable in $p_X(x)^+$
			with integral equal to 1,
			$$\int_{p_X(x)^+} {p_X(x)} dx = \int_{-\infty}^{+\infty} {p_X(x)} dx= 1.$$
			\item Instead of smoothness we will require that the derivative $p_X'(x)$ exists and is finite almost everywhere in the Lebesgue measure.
		\end{enumerate}
	\end{definition}

	Following Nielsen \cite{Nielsen2021} we will prove some preliminary theorems.

	\begin{theorem}
		Given the random variable $X$ with cumulative distribution function $F_X(x)$ and probability density $p_X(x)$, consider the random variable $Y \sim s X + l$ (for some $s \in (0,+\infty)$ and $l \in \mathbb{R}$) with cumulative distribution function $F_Y(y)$ and probability density $p_Y(y)$. Then

		$$
		F_Y(\bar{y}) = F_X \bigg(\frac{\bar{y}-l}{s}\bigg), \forall \bar{y} \in \mathbb{R}
		$$
		and
		$$
		p_Y(y) = \frac{1}{s} p_X \bigg( \frac{y-l}{s}\bigg).
		$$
	\end{theorem}
	\begin{proof}
		It can be observed that by defining $Y \sim s X + l$, this random variable has cumulative distribution function
		$$
		F_Y(\bar{y}) = P(Y \leq \bar{y}) = P((sX+l) \leq \bar{y}) = P\bigg(X \leq \frac{\bar{y}-l}{s} \bigg) = F_X\bigg( \frac{\bar{y}-l}{s} \bigg).
		$$
		Then
		$$
		\int_{-\infty}^{\bar{y}} {p_Y(y)} dy = F_Y(\bar{y}) = F_X\bigg( \frac{\bar{y}-l}{s} \bigg) = \int_{-\infty}^{\frac{\bar{y}-l}{s}} {p_X(x)} dx = \int_{-\infty}^{\bar{y}} {\frac{1}{s} p_X\bigg(\frac{y-l}{s}\bigg)} dy\,\,,
		$$
		so that
		$$
		p_Y(y) = \frac{1}{s} p_X \bigg( \frac{y-l}{s}\bigg).
		$$
	\end{proof}

	Now, if we define the group $G$ of transformations
	$$
	g_{l,s} \cdot p(x) = \frac{1}{s} p \bigg(\frac{x-l}{s} \bigg),
	$$
	a family of probability distributions is induced as the orbit of the group $G$ acting on a generatrix probability distribution $p(x)$. We may thus define the corresponding location-scale family as follows:
	\begin{definition}
		Given a generatrix function $p(x)$, the \emph{location-scale family} generated from the generatrix $p(x)$ is defined as
		$$
		\mathcal{F}_p = G \cdot p(x) = \big\{ g \cdot p \mid g\in G\big\} = \bigg\{ g_{l,s} \cdot p(x) = \frac{1}{s} p \bigg(\frac{x-l}{s} \bigg) \mid s,l \in \mathbb{R}, s>0 \bigg\}
		$$
	\end{definition}

	\begin{definition}
		When the generatrix $p(x)$ is a non-smooth arbitrary support distribution then the location scale family $\mathcal{F}_p $ is a \emph{non-smooth location-scale family with arbitrary support}.
	\end{definition}

	In \cite{Nielsen2021} a characterization of this problem in the case of a smooth $p(x)$ with $p(x)^+ = (-\infty,+\infty)$ is proved. In the present work this characterization is generalized as follows:

	\begin{theorem}
		Given a non-smooth arbitrary support distribution $p(x)$ and the non-smooth location-scale family with arbitrary support $\mathcal{F}_p $, then the metric tensor of the probability family manifold is
		$$
		g_{ij} (l,s) = \frac{1}{s^2} \begin{pmatrix}
			a^2 & c \\
			c & b^2
		\end{pmatrix} \,\,,
		$$
		where
		$$
		a^2 = \mathbb{E}_{p} \Bigg[ \bigg(\frac{p'(x)}{p(x)} \bigg)^2 \Bigg] \geq 0 \,\,,
		$$

		$$
		b^2 = \mathbb{E}_{p}  \Bigg[ \bigg(1+x\frac{p'(x)}{p(x)} \bigg)^2 \Bigg] \geq 0 \,\,,
		$$

		$$
		c = \mathbb{E}_{p} \Bigg[ \frac{p'(x)}{p(x)} \bigg(1+x\frac{p'(x)}{p(x)} \bigg) \Bigg]\,\,.
		$$
	\end{theorem}

	\begin{proof}
		By defining $\lambda= (l,s)$ and
		$$
		p_{\lambda}(x) = \frac{1}{s}p \bigg( \frac{x-l}{s}\bigg)
		$$
		we can compute the metric tensor of the manifold
		$$
		g_{ij}(\lambda) = I_{F} (\lambda) = \mathbb{E}_{p_{\lambda}} \big[ - \nabla \nabla \log p_{\lambda}(x) \big],
		$$
		where the expected value for a general function $f(x)$ is defined as
		$$
		\mathbb{E}_{p_{\lambda}} \big[ f(x) \big] = \int_{p_{\lambda}(x)^+} {p_{\lambda}(x)f(x)} dx \,\,.
		$$
		We observe that, when $f(x)$ is defined in all $\mathbb{R}$, this definition is equivalent to the usual definition
		$$
		\int_{-\infty}^{+\infty} {p_{\lambda}(x)f(x)} dx.
		$$
		With this definition we can repeat the the steps of the proof in Nielsen \cite{Nielsen2021}. Specifically, using the expression
		$$
		\mathbb{E}_{p_{\lambda}} \big[ - \nabla \nabla \log p_{\lambda}(x) \big] = \mathbb{E}_{p_{\lambda}} \big[ - \partial_i \partial_j \log p_{\lambda}(x) \big]= \mathbb{E}_{p_{\lambda}} \big[ \partial_i \log p_{\lambda}(x) \partial_j \log p_{\lambda}(x) \big] \,\,,
		$$
		we can compute all terms, beginning with $g_{11} (\lambda)$:
		$$
		g_{11}(\lambda) =  \mathbb{E}_{p_{\lambda}} \Bigg[ \bigg( \frac{p'_{\lambda}(x)}{p_{\lambda}(x)}\bigg)^2 \Bigg] =  \int_{p_{\lambda}(x)^+} {p_{\lambda}(x) \bigg( \frac{p'_{\lambda}(x)}{p_{\lambda}(x)}\bigg)^2 } dx \,\,.
		$$
		It can be observed that $\Big( \frac{p'_{\lambda}(x)}{p_{\lambda}(x)}\Big)^2 $ is not defined where $p_{\lambda}(x)=0$, but, using our formalization,  integration is not performed over those regions. Moreover, since $p_{\lambda}'(x)$ exists almost everywhere, the integral can be computed even if the function $p_{\lambda}(x)$ is non-smooth. \\
		In \cite{Nielsen2021} is then observed that
		$$
		g_{11} (\lambda) = \mathbb{E}_{p_{\lambda}} \Bigg[ \bigg( \frac{p'_{\lambda}(x)}{p_{\lambda}(x)}\bigg)^2 \Bigg] =  \int_{p_{\lambda}(x)^+} {p_{\lambda}(x) \bigg( \frac{p_{\lambda}'(x)}{p_{\lambda}(x)}\bigg)^2 } dx =
		$$

		$$
		= \int_{p\big(\frac{x-l}{s}\big)^+} {\frac{1}{s} p\bigg(\frac{x-l}{s}\bigg) \Bigg( \frac{1}{s}\frac{p'(\frac{x-l}{s})}{p(\frac{x-l}{s})}\Bigg)^2 } dx =
		$$

		$$
		= \frac{1}{s^2} \int_{p(x)^+} {p(x) \bigg( \frac{p'(x)}{p(x)}\bigg)^2 } dx = \frac{1}{s^2} \mathbb{E}_{p} \Bigg[ \bigg( \frac{p'(x)}{p(x)}\bigg)^2 \Bigg] \,\,.
		$$

		In the same way it can be proved that
		$$
		g_{22} (\lambda) = \frac{1}{s^2}\mathbb{E}_{p}  \Bigg[ \bigg(1+x\frac{p'(x)}{p(x)} \bigg)^2 \Bigg] = \frac{1}{s^2}\int_{p(x)^+} {p(x)\bigg(1+x\frac{p'(x)}{p(x)} \bigg)^2} dx
		$$
		and
		$$
		g_{12} (\lambda) = \frac{1}{s^2} \mathbb{E}_{p} \Bigg[ \frac{p'(x)}{p(x)} \bigg(1+x\frac{p'(x)}{p(x)} \bigg) \Bigg] = \frac{1}{s^2} \int_{p(x)^+} {p(x) \frac{p'(x)}{p(x)} \bigg(1+x\frac{p'(x)}{p(x)} \bigg) } dx.
		$$

		In conclusion, the matrix can be rewritten in the form
		$$
		g_{ij} (\lambda) = g_{ij} (l,s)= \frac{1}{s^2}
		\begin{pmatrix}
			a^2 & c \\
			c & b^2
		\end{pmatrix},
		$$
		where
		$$
		a^2 = \mathbb{E}_{p} \Bigg[ \bigg(\frac{p'(x)}{p(x)} \bigg)^2 \Bigg] \geq 0 \,\,,
		$$
		$$
		b^2 = \mathbb{E}_{p}  \Bigg[ \bigg(1+x\frac{p'(x)}{p(x)} \bigg)^2 \Bigg] \geq 0 \,\,,
		$$
		$$
		c = \mathbb{E}_{p} \Bigg[ \frac{p'(x)}{p(x)} \bigg(1+x\frac{p'(x)}{p(x)} \bigg) \Bigg] \,\,.
		$$
	\end{proof}

	In Nielsen \cite{Nielsen2021} it is observed that for a smooth generatrix $p(x)$ with support $p(x)^+ = (\infty,+\infty)$, if the function is even, i.e. $p(x)=p(-x)$, then the coefficient $c=\mathbb{E}_{p} \Bigg[ \frac{p'(x)}{p(x)} \bigg(1+x\frac{p'(x)}{p(x)} \bigg) \Bigg]$ is equal to zero, which implies that defining $\lambda= (l,s)$ then
	$$
	g_{ij}(\lambda) = \frac{1}{s^2} \begin{pmatrix}
		a^2 & 0 \\
		0 & b^2
	\end{pmatrix} \,\,.
	$$
	By using the change of coordinates $\theta(\lambda) = \big( \frac{a}{b} \lambda_1,\lambda_2 \big)$ (with $a=\sqrt{a^2}$ and $b=\sqrt{b^2}$), the metric tensor becomes
	$$ g_{ij}(\theta) = \frac{b^2}{\theta_2^2}
	\begin{pmatrix}
		1 & 0 \\
		0 & 1
	\end{pmatrix} \,\,.
	$$
	The above is the Poincaré hyperbolic half plane model metric: from this it can be concluded that the curvature is
	$$
	S = -\frac{1}{b^2} \,\,.
	$$
	In the present work paper this result is generalized to non-even $p(x)$.

	\subsection{Classification of non-even distributions}
	In the general case
	\begin{theorem}
		Given a non-smooth arbitrary support distribution $p(x)$ and the non-smooth location-scale family with arbitrary support $\mathcal{F}_p $, the curvature has value
		$$
		S = - \frac{a^2}{a^2b^2 -c^2}\,\,,
		$$
		where
		$$
		a^2 = \mathbb{E}_{p} \Bigg[ \bigg(\frac{p'(x)}{p(x)} \bigg)^2 \Bigg] \geq 0 \,\,,
		$$
		$$
		b^2 = \mathbb{E}_{p}  \Bigg[ \bigg(1+x\frac{p'(x)}{p(x)} \bigg)^2 \Bigg] \geq 0 \,\,, and
		$$
		$$
		c = \mathbb{E}_{p} \Bigg[ \frac{p'(x)}{p(x)} \bigg(1+x\frac{p'(x)}{p(x)} \bigg) \Bigg] \,\,.
		$$
		It can be observed that if $c=0$ (and $a \neq 0$)
		$$
		S = - \frac{a^2}{a^2b^2} = - \frac{1}{b^2}
		$$
		as in \cite{Nielsen2021}.
	\end{theorem}

	\begin{proof}
		The calculations have been perfomed via CAS and then checked manually, they are omitted for brevity.
	\end{proof}

	A more general result follows:

	\begin{theorem}
		Given a non-smooth arbitrary support distribution $p(x)$ and the non-smooth location-scale family with arbitrary support $\mathcal{F}_p $, the curvature (where defined) is negative:
		$$
		S = - \frac{a^2}{a^2b^2 -c^2} < 0 \,\,.
		$$
	\end{theorem}
	\begin{proof}
		Studying the sign of
		$$
		S = - \frac{a^2}{a^2b^2 -c^2} \,\,,
		$$
		we observe that the numerator of $S$ is positive (since $a^2$ is the expected value of a positive function) and we only have to determine the sign of the quantity $a^2 b^2-c^2$. We define the auxiliary functions
		$$
		A(x) = \frac{p'(x)}{p(x)} \,\,,
		$$
		$$
		B(x) = \bigg(1+x\frac{p'(x)}{p(x)}\bigg) \,\,,
		$$
		using which the coefficients become
		$$
		a^2 = \mathbb{E}_p \big[ A(x)^2 \big] = \int_{p(x)^+} {p(x) A(x)^2} dx \,\,,
		$$
		$$
		b^2 = \mathbb{E}_p \big[ B(x)^2 \big] = \int_{p(x)^+} {p(x) B(x)^2} dx \,\,, and
		$$
		$$
		c = \mathbb{E}_p [A(x)B(x)] = \int_{p(x)^+} {p(x) A(x)B(x)} dx \,\,.
		$$

		Proceeding as in the proof of the Cauchy-Schwarz inequality, defining
		$$
		\alpha= \frac{\int_{p(x)^+} {p(x) A(x)B(x)} dx}{\int_{p(x)^+} {p(x) A(x)^2} dx} = \frac{c}{a^2}
		$$
		we observe that
		$$
		\int_{p(x)^+} {p(x) [B(x) - \alpha A(x)]^2} dx \geq 0 \,\,,
		$$
		then
		$$
		\int_{p(x)^+} {p(x) [B(x) - \alpha A(x)]^2} dx  = \int_{p(x)^+} {p(x) [B(x)^2 - 2 \alpha A(x) B(x) + \alpha^2 A(x)^2]} dx =
		$$
		$$
		= \int_{p(x)^+} {p(x) [B(x)^2]} dx - 2 \alpha \int_{p(x)^+} {p(x) [A(x)B(x)]} dx + \alpha^2 \int_{p(x)^+} {p(x) [A(x)^2]} dx =
		$$
		$$
		= b^2 - 2 \frac{c}{a^2} c + \bigg(\frac{c}{a^2} \bigg)^2 a^2 = b^2 - \frac{c^2}{a^2} \geq 0 \,\,.
		$$

		Supposing $a^2 \neq 0$, by multiplying both sides of the previous expression by  $a^2 \geq 0$ we obtain
		$$
		a^2 b^2 - c^2 \geq 0 \,\,,
		$$
		from which follows that
		$$
		S = - \frac{a^2}{a^2b^2 -c^2} \leq 0 \,\,.
		$$
		This means that (when non-degenerate) the curvature of two-dimensional location-scale families is negative, and the corresponding manifolds are locally hyperbolic.\\\\

		It is of interest at this point to analyze the cases of flat geometry ($S=0$, implying $a^2 = 0$) and degeneracy ($a^2b^2 -c^2=0$, implying in the limit that $S=\infty$).\\\\

		Supposing $S=0$ is equivalent to supposing
		$$
		a^2 = \int_{p(x)^+} {p(x) [A(x)^2]} dx = \int_{p(x)^+} {p(x) \bigg(\frac{p'(x)}{p(x)} \bigg)^2} dx = \int_{p(x)^+} {\frac{p'(x)^2}{p(x)} } dx  = 0 \,\,.
		$$
		In conclusion, since $p(x)>0$ and $p'(x)^2\geq 0$ the above condition is equivalent to
		$$
		p'(x) = 0, \forall x \in p(x)^+ \,\,,
		$$
		which is a well known differential equation with the well known smooth solution
		$$
		p(x) = k, \forall x \in p(x)^+
		$$
		for some real constant $k \in \mathbb{R}$.
		Now, if the support $p(x)^+$ has infinite Lebesgue measure, such a distribution cannot exist (any such function would not integrate to a finite value). But even in the case where the support had finite measure we would have
		$$
		c = \int_{p(x)^+} {p(x) \frac{p'(x)}{p(x)} \bigg(1+x\frac{p'(x)}{p(x)} \bigg) } dx =0
		$$
		making the metric tensor
		$$ g_{ij}(\lambda) =
		\begin{pmatrix}
			0 & 0 \\
			0 & b^2
		\end{pmatrix} \,\,,
		$$
		hence of rank $1$, hence this particular density would correspond to a singular point of the manifold (independently of the smoothness of the solution). Such a case will therefore not be admissible.\\\\

		Considering now the condition $a^2b^2 -c^2 = 0$, if $a^2 \neq 0$ this is clearly equivalent to
		$$
		\int_{-\infty}^{+\infty} {p(x) [B(x) - \alpha A(x)]^2} dx = 0 \,\,,
		$$
		which implies
		$$
		B(x)-\alpha A(x)  = 0, \forall x \in \mathbb{R}\,\,.
		$$
		We must thus solve the differential equation
		$$
		\alpha \frac{p'(x)}{p(x)} - \bigg( 1 + x \frac{p'(x)}{p(x)} \bigg) = 0 \,\,,
		$$
		which can be rearranged as
		$$
		(\alpha - x) \frac{p'(x)}{p(x)} = 1 \,\,.
		$$
		We first look for a continuous solution: in this case the equation can be solved using the differentials method:
		$$
		\int \frac{dp}{p} = \int \frac{dx}{ \big(\alpha - x \big) }
		$$
		$$
		\log p = - \log  \big(\alpha - x \big) + k
		$$
		where $k$ is an arbitrary constant. Passing to the exponential
		$$
		p(x) = \frac{e^k}{\big(\alpha - x \big) } \,\,.
		$$
		We observe that this distribution could potentially be well defined over a finite-measure subset of $(-\infty,\alpha)$. We may suppose, without loss of generality, that the support of the distribution is of the form $(\alpha-2 \varepsilon, \alpha- \varepsilon)$ for arbitrary real numbers $\varepsilon, \alpha >0$:
		$$
		p_{\alpha}(x) =
		\begin{cases}
			\frac{1}{K(\alpha,\varepsilon)}\frac{1}{(\alpha - x ) }, \text{ if } x \in [\alpha-2\varepsilon, \alpha- \varepsilon] \\
			0, \text{ othewrise}
		\end{cases} \,\,,
		$$
		where
		$$
		K(\alpha,\varepsilon) = \int_{\alpha-2\varepsilon}^{\alpha- \varepsilon} {\frac{1}{\alpha-x}} dx >0 \,\,.
		$$
		It can be observed that, independently from the smoothness of the solution, the following equality holds pointwise:
		$$
		\frac{p'_{\alpha}(x)}{p_{\alpha}(x)} = \frac{1}{\alpha-x} , \forall x \in [\alpha-2\varepsilon, \alpha- \varepsilon] \,\,,
		$$
		which implies
		$$
		1+x \frac{p'_{\alpha}(x)}{p_{\alpha}(x)} = \alpha \frac{p'_{\alpha}(x)}{p_{\alpha}(x)} \,\,.
		$$
		The curvature coefficients can therefore be computed as
		$$
		a^2 = \int_{\alpha-2\varepsilon}^{\alpha- \varepsilon} {p_{\alpha}(x) \bigg( \frac{p'_{\alpha}(x)}{p_{\alpha}(x)} \bigg)^2 } dx \,\,,
		$$
		$$
		b^2 = \int_{\alpha-2\varepsilon}^{\alpha- \varepsilon} {p_{\alpha}(x)  \bigg( 1+x \frac{p'_{\alpha}(x)}{p_{\alpha}(x)}\bigg)^2 } dx = \alpha^2 \int_{\alpha-2\varepsilon}^{\alpha- \varepsilon} {p_{\alpha}(x)  \bigg( \frac{p'_{\alpha}(x)}{p_{\alpha}(x)} \bigg)^2 } dx = \alpha^2 a^2 \,\,, and
		$$
		$$
		c = \int_{\alpha-2\varepsilon}^{\alpha- \varepsilon} {p_{\alpha}(x) \frac{p'_{\alpha}(x)}{p_{\alpha}(x)} \bigg( 1+x \frac{p'_{\alpha}(x)}{p_{\alpha}(x)}\bigg) } dx = \alpha \int_{\alpha-2\varepsilon}^{\alpha- \varepsilon} {p_{\alpha}(x)  \bigg( \frac{p'_{\alpha}(x)}{p_{\alpha}(x)} \bigg)^2 } dx = \alpha a^2 \,\,.
		$$
		The curvature then takes the form
		$$
		S = - \frac{a^2}{a^2b^2-c^2} = - \frac{a^2}{a^2 \alpha^2 a^2 - \alpha^2 (a^2)^2} = - \frac{a^2}{0} = -\infty \,\,,
		$$
		but also in this case it can be observed that the metric tensor is degenerate, being of rank 1 (instead of rank 2).
	\end{proof}

	\medskip

	In conclusion, it can be claimed that in the general case of 2-dimensional location-scale families, when no singularities occur, the curvature
	$$
	S = - \frac{a^2}{a^2b^2 -c^2} < 0 \,\,.
	$$
	This generalizes the result of Nielsen \cite{Nielsen2021} to non-even and non-smooth $p(x)$ with arbitrary support. This result has wide implications because most of experimentally used distributions constitute location-scale families and we have proved that in the 2-dimensional case (also frequently occurring) they are locally hyperbolic in the Information Geometry sense.

	\subsection{Examples of location-scale families with arbitrary support}
	In the following is a brief list of well known examples belonging to the family of 2-dimensional location-scale distributions with arbitrary support.

	\begin{enumerate}
		\item The Gaussian Normal distribution:
		$$
		p(x|\mu,\sigma^2) = \frac{1}{\sqrt{2 \pi \sigma^2}} e^{- \frac{(x-\mu)^2}{2 \sigma^2}}
		$$
		with generatrix
		$$
		p(x) = \frac{1}{\sqrt{\pi}} e^{-x^2} \,\,.
		$$
		This function is also a classical location-scale function, being smooth and having support $p(x)^+ = (-\infty,+\infty)$. It can be observed that since $p(x)$ is even, then the curvature is in the form
		$$
		S = - \frac{1}{b^2} \,\,.
		$$
		By computing $b^2$ as
		$$
		b^2 = \int_{p(x)^+} {p(x) \bigg(1+x\frac{p'(x)}{p(x)} \bigg)^2} dx = \int_{-\infty}^{+\infty} {\frac{1}{\sqrt{\pi}} e^{-x^2} \Bigg(1 + x \frac{\frac{-2x}{\sqrt{\pi}} e^{-x^2} }{\frac{1}{\sqrt{\pi}} e^{-x^2} } \Bigg)^2 } dx =
		$$
		$$
		=  \int_{-\infty}^{+\infty} {\frac{1}{\sqrt{\pi}} e^{-x^2} (1-2x^2)^2} dx = \int_{-\infty}^{+\infty} {\frac{1}{\sqrt{\pi}} e^{-x^2} (1-4x^2 + 4x^4)} dx = 2 \,\,,
		$$
		the curvature in this case is
		$$
		S = - \frac{1}{b^2} = - \frac{1}{2} \,\,.
		$$

		\bigskip

		\item The Cauchy distribution
		$$
		p(x|x_0,\gamma) = \frac{1}{\pi} \frac{\gamma}{(x-x_0)^2+\gamma^2}
		$$
		with generatrix
		$$
		p(x) = \frac{1}{\pi} \frac{1}{x^2+1}
		$$
		is also a classical location-scale function, being smooth and with support $p(x)^+ =(-\infty,+\infty)$.
		To find the curvature we observe that also in this case the function $p(x)$ is even, therefore only the coefficient $b^2$ will be computed:
		$$
		b^2 = \int_{p(x)^+} {p(x) \bigg(1+x\frac{p'(x)}{p(x)} \bigg)^2} dx = \int_{-\infty}^{+\infty} {\frac{1}{\pi} \frac{1}{1+x^2} \Bigg(1+x\frac{\frac{-2x}{\pi (1+x^2)^2}}{\frac{1}{\pi (1+x^2)}} \Bigg)^2} dx =
		$$
		$$
		= \int_{-\infty}^{+\infty} {\frac{1}{\pi} \frac{(1-x^2)^2}{(1+x^2)^3}} dx = \Bigg[ \frac{1}{\pi} \frac{(x^2+1)^2 \tan^{-1}(x)- x^3 +x}{2 (x^2+1)^2}\Bigg|_{-\infty}^{+\infty} = \frac{1}{2} \,\,.
		$$
		The last integral was computed by means of Wolfram Alpha. The curvature is thus
		$$
		S = - \frac{1}{b^2} = - 2 \,\,,
		$$
		once again negative.

		\bigskip

		\item The general exponential distribution
		$$
		p(x|\lambda,\zeta) =
		\begin{cases}
			\lambda e^{-\lambda(x-\zeta)}, \text{ if } x \geq \zeta \\
			0, \text{ otherwise}
		\end{cases}
		$$
		with generatrix
		$$
		p(x) =
		\begin{cases}
			e^{-x}, \text{ if } x \geq 0 \\
			0, \text{ otherwise}
		\end{cases}
		$$
		is not a classical location-scale family since the generatrix function has support $p(x)^+ = (0,+\infty)$, is not differentiable in $x=0$ and is not even.
		To find the curvature in this case we will needed to compute all terms:
		$$
		a^2 = \int_{p(x)^+} {p(x) \bigg(\frac{p'(x)}{p(x)} \bigg)^2} dx = \int_{0}^{+\infty} {e^{-x} \Bigg(\frac{- e^{-x}}{e^{-x}} \Bigg)^2} dx = \int_{0}^{+\infty} {e^{-x}} dx = 1 \,\,,
		$$
		$$
		b^2 = \int_{p(x)^+} {p(x) \bigg(1+x\frac{p'(x)}{p(x)} \bigg)^2} dx = \int_{0}^{+\infty} {e^{-x} \Bigg(1+x\frac{-e^{-x}}{e^{-x}} \Bigg)^2} dx =
		$$
		$$
		= \int_{0}^{+\infty} {e^{-x} (1-x)^2} dx = \int_{0}^{+\infty} {e^{-x} \big(1-2x+x^2 \big)} dx = 1\,\,,
		$$
		$$
		c = \int_{p(x)^+} {p(x) \frac{p'(x)}{p(x)} \bigg(1+x\frac{p'(x)}{p(x)} \bigg)} dx = \int_{0}^{+\infty} {e^{-x} \frac{-e^{-x}}{e^{-x}} \bigg(1+x\frac{-e^{-x}}{e^{-x}} \bigg)} dx =
		$$
		$$
		= - \int_{0}^{+\infty} {e^{-x} (1-x)} dx = 0 \,\,.
		$$
		Then the curvature is
		$$
		S  = - \frac{a^2}{a^2 b^2 - c^2} = - \frac{1}{1 -0} = -1 \,\,.
		$$

		\bigskip

		\item The Laplace distribution
		$$
		p(x|\mu,b) = \frac{1}{2b} e^{-\frac{|x-\mu|}{b}}
		$$
		with generatrix
		$$
		p(x) = \frac{1}{2} e^{-|x|}
		$$
		is also not a classical location-scale family, since the generatrix function is not differentiable in $x=0$ even though the support is $\mathbb{R}$,  $p(x)^+ = (-\infty,+\infty)$.
		In this case the function $p(x)$ is even, and it will be sufficient to compute only the coefficient $b^2$:
		$$
		b^2 = \int_{p(x)^+} {p(x) \bigg(1+x\frac{p'(x)}{p(x)} \bigg)^2} dx = \int_{-\infty}^{+\infty} {\frac{1}{2} e^{-|x|} \bigg(1 - x \frac{|x|}{x} \bigg)^2} dx =
		$$
		$$
		= \int_{0}^{+\infty} {e^{-x} (1-x)^2} dx = \int_{0}^{+\infty} {e^{-x} \big(1-2x +x^2 \big)} dx = 1 \,\,.
		$$
		The curvature is then
		$$
		S = -\frac{1}{b^2} = -1 \,\,.
		$$
	\end{enumerate}

	\bigskip

	We may observe that, following Strimmer \cite{Strimmer2022}, the curvature of the manifold is connected to Entropy: the higher the absolute value of the curvature, the more precise will be the maximum-likelihood estimation, but less general will be the distribution. From our few examples it can be seen that the Gaussian family, which can be proved to be the maximum-entropy distribution family with prescribed (mean and) variance over $(-\infty,+\infty)$), has absolute curvature $|S| = \frac{1}{2}$, in contrast with the other distributions for which we have computed a curvature $|S|>\frac{1}{2}$. This suggests that the Gaussian is in fact the flattest location-scale distribution family with prescribed (mean and) variance over $(-\infty,+\infty)$ in the case $d = 2$.

	\section{Exponential families}
	\subsection{New flatness criteria for exponential families}
	In Amari \cite{Amari2016} we find among other things the definition of an exponential family of distributions:
	\begin{definition}
		Given a point $x \in \mathbb{R}^d$ and a parameter vector $\theta= (\theta_1,..., \theta_d) \in \mathbb{R}^d$ of dimension $d$, by defining the functions $h(x) = (h_1(x), ..., h_d(x))$, $k(x)$ and $\psi(\theta)$ it is possible to define the distribution of the exponential family as
		$$
		p(x|\theta) = e^{\theta \cdot h(x) + k(x) - \psi(\theta)} \,\,,
		$$
		where the scalar function $\psi(\theta)$ is called the \emph{regularization function} and can be proved to be convex.
	\end{definition}

	In Amari \cite{Amari2016} then the following is proved:
	\begin{theorem}
		Given an exponential family the manifold metric tensor is
		$$
		g_{ij} (\theta)  = I_F(\theta) = \mathbb{E} \big[ - \nabla \nabla \log p(x|\theta) \big] = \nabla \nabla \psi (\theta) \,\,,
		$$
		where $\nabla \nabla \psi = \big( \partial_i \partial_j \psi \big)$ is the Hessian of the function $\psi(\theta)$ and $\partial_i= \frac{\partial}{\partial \theta_i}$ are the partial derivatives with respect to the components of the parameter.
	\end{theorem}

	From this point onward we will restrict our attention to dimension 2. Supposing thus $\theta  = (\theta_1, \theta_2)$, the metric tensor becomes
	$$
	g_{ij} (\theta) =
	\begin{pmatrix}
		\partial_1^2 \psi (\theta) & \partial_1 \partial_2 \psi (\theta)  \\
		\partial_2 \partial _1 \psi (\theta)  & \partial_2^2 \psi (\theta)
	\end{pmatrix}
	$$
	where $\partial_1 = \frac{\partial}{\partial \theta_1}$ and $\partial_2 = \frac{\partial}{\partial \theta_2}$.

	With the aid of a CAS a novel formula for the curvature in the case of Exponential Families can be more easily derived:

	\begin{theorem}
		For an Exponential family with metric tensor of the form
		$$
		g_{ij} (\theta) =
		\begin{pmatrix}
			g_{11} (\theta)  & g_{12} (\theta)  \\
			g_{21} (\theta)  & g_{22} (\theta)
		\end{pmatrix}
		=
		\begin{pmatrix}
			\partial_1^2 \psi (\theta)  & \partial_1 \partial_2 \psi (\theta)  \\
			\partial_2 \partial _1 \psi (\theta)  & \partial_2^2 (\theta)  \psi
		\end{pmatrix} \,\,,
		$$
		the curvature can be computed by the formula
		$$
		S (\theta)  = \frac{\left| \begin{matrix}
				g_{11} & \partial_2 g_{11}  & \partial_1 g_{11}  \\
				g_{12}  & \partial_2 g_{12}  & \partial_1 g_{12}  \\
				g_{22}  & \partial_2 g_{22}  & \partial_1 g_{22}
			\end{matrix} \right| }{4 \left| \begin{matrix}
				g_{11} & g_{12}  \\
				g_{21}  & g_{22}
			\end{matrix}\right|^2} \,\,.
		$$
	\end{theorem}

	\begin{proof}
		By direct computation of the curvature in this rather general case we have
		$$
		S (\theta)  = \frac{1}{4} \frac {-\partial_1^2\psi \partial_2^3\psi \partial_1^2 \partial_2 \psi + \partial_1^2 \psi (\partial_1 \partial_2^2 \psi)^2
		}
		{(\partial_1^2 \psi)^2 (\partial_2^2 \psi)^2 - 2 \partial_1^2 \psi \partial_2^2 \psi (\partial_1 \partial_2 \psi)^2 + (\partial_1 \partial_2 \psi)^4
		}+
		$$
		$$
		\frac{1}{4} \frac{	- \partial_1^3\psi \partial_2^2 \psi \partial_1 \partial_2^2 \psi 	+ \partial_2^2 \psi (\partial_2 \partial_1^2 \psi)^2
		}{(\partial_1^2 \psi)^2 (\partial_2^2 \psi)^2 - 2 \partial_1^2 \psi \partial_2^2 \psi (\partial_1 \partial_2 \psi)^2 + (\partial_1 \partial_2 \psi)^4
		}+
		$$
		$$
		\frac{1}{4} \frac{- \partial_1 \partial_2 \psi \partial_1 \partial_2^2 \psi \partial_1^2 \partial_2 \psi + \partial_1^3 \psi \partial_2^3 \psi \partial_1 \partial_2 \psi }{{(\partial_1^2 \psi)^2 (\partial_2^2 \psi)^2 - 2 \partial_1^2 \psi \partial_2^2 \psi (\partial_1 \partial_2 \psi)^2 + (\partial_1 \partial_2 \psi)^4
		}} =
		$$
		$$
		= \frac{1}{4} \frac {-\partial_1^2\psi \partial_2^3\psi \partial_1^2 \partial_2 \psi + \partial_1^2 \psi (\partial_1 \partial_2^2 \psi)^2
		}
		{\big(\partial_1^2 \psi \partial_2^2 \psi - (\partial_1 \partial_2 \psi)^2 \big)^2
		}+
		$$
		$$
		\frac{1}{4} \frac{	- \partial_1^3\psi \partial_2^2 \psi \partial_1 \partial_2^2 \psi 	+ \partial_2^2 \psi (\partial_2 \partial_1^2 \psi)^2
		}{\big(\partial_1^2 \psi \partial_2^2 \psi - (\partial_1 \partial_2 \psi)^2 \big)^2
		}+
		$$
		$$ \frac{1}{4} \frac{- \partial_1 \partial_2 \psi \partial_1 \partial_2^2 \psi \partial_1^2 \partial_2 \psi + \partial_1^3 \psi \partial_2^3 \psi \partial_1 \partial_2 \psi }{\big(\partial_1^2 \psi \partial_2^2 \psi - (\partial_1 \partial_2 \psi)^2 \big)^2
		} .$$

		We observe that the numerator
		$$ \Big(\partial_1^2 \psi \partial_2^2 \psi - (\partial_1 \partial_2 \psi)^2 \Big)^2 = \Big(g_{11} g_{22} - g_{12}^2 \Big)^2 = \left| \begin{matrix}
			g_{11} & g_{12} \\
			g_{21} & g_{22}
		\end{matrix} \right|^2 \geq 0 \,\,. $$
		Considering the polynomial form at numerator
		$$
		-\partial_1^2\psi \partial_2^3\psi \partial_1^2 \partial_2 \psi + \partial_1^2 \psi (\partial_1 \partial_2^2 \psi)^2 - \partial_1^3\psi \partial_2^2 \psi \partial_1 \partial_2^2 \psi 	+ \partial_2^2 \psi (\partial_2 \partial_1^2 \psi)^2 +
		$$
		$$
		- \partial_1 \partial_2 \psi \partial_1 \partial_2^2 \psi \partial_1^2 \partial_2 \psi + \partial_1^3 \psi \partial_2^3 \psi \partial_1 \partial_2 \psi \,\,,
		$$
		and consider the following equivalences
		$$
		\partial_1 \partial_2 \psi= g_{12} = g_{21} = \partial_2 \partial_1 \psi\,\,,
		$$
		$$
		\partial_1^3 \psi = \partial_1 g_{11}\,\,,
		$$
		$$
		\partial_2^3 \psi = \partial_2 g_{22}\,\,,
		$$
		$$
		\partial_2 \partial_1^2 \psi = \partial_2 g_{11} = \partial_1 g_{12} = \partial_1 g_{21} \,\,,
		$$
		$$
		\partial_1 \partial_2^2 \psi = \partial_1 g_{22} = \partial_2 g_{12} = \partial_2 g_{21} \,\,.
		$$
		Then the polynomial can be rewritten as
		$$
		g_{11} (\partial_2 g_{12} \partial_1 g_{22} - \partial_2 g_{22} \partial_1 g_{12}) + g_{22} (\partial_2 g_{11} \partial_1 g_{12} - \partial_1 g_{11} \partial_2 g_{12})+
		$$
		$$
		-g_{12} (\partial_2 g_{11} \partial_1 g_{22} -\partial_1 g_{11} \partial_2 g_{22}) =
		$$
		(when expressed in terms of determinants)
		$$
		= g_{11} \left|
		\begin{matrix}
			\partial_2 g_{12} & \partial_1 g_{12} \\
			\partial_2 g_{22} & \partial_1 g_{22}
		\end{matrix}\right| + g_{22} \left| \begin{matrix}
			\partial_2 g_{11} & \partial_1 g_{11} \\
			\partial_2 g_{22} & \partial_1 g_{12}
		\end{matrix}\right| - g_{12} \left| \begin{matrix}
			\partial_2 g_{11} & \partial_1 g_{11} \\
			\partial_2 g_{22} & \partial_1 g_{22}
		\end{matrix}\right| =
		$$
		$$
		= \left|
		\begin{matrix}
			g_{11} & \partial_2 g_{11} & \partial_1 g_{11} \\
			g_{12} & \partial_2 g_{12} & \partial_1 g_{12} \\
			g_{22} & \partial_2 g_{22} & \partial_1 g_{22}
		\end{matrix} \right| = \left| \begin{matrix}
			\partial_1^2 \psi & \partial_1^2 \partial_2 \psi & \partial_1^3 \psi \\
			\partial_1 \partial_2 \psi & \partial_1 \partial_2^2 \psi & \partial_1^2 \partial_2 \psi \\
			\partial_2^2 \psi & \partial_2^3 \psi & \partial_1 \partial_2^2 \psi \\
		\end{matrix} \right| \,\,,
		$$
		which implies
		$$
		S (\theta)  = \frac{\left| \begin{matrix}
				g_{11} & \partial_2 g_{11}  & \partial_1 g_{11}  \\
				g_{12}  & \partial_2 g_{12}  & \partial_1 g_{12}  \\
				g_{22}  & \partial_2 g_{22}  & \partial_1 g_{22}
			\end{matrix} \right| }{4 \left| \begin{matrix}
				g_{11} & g_{12}  \\
				g_{21}  & g_{22}
			\end{matrix}\right|^2} \,\,.$$
	\end{proof}

	From the previous formula some interesting  criteria for flatness of exponential families in dimension 2 follow:
	\begin{theorem}
		Observing that the curvature in the case of exponential families can be computed as
		$$
		S (\theta)  = \frac{\left| \begin{matrix}
				g_{11} & \partial_2 g_{11}  & \partial_1 g_{11}  \\
				g_{12}  & \partial_2 g_{12}  & \partial_1 g_{12}  \\
				g_{22}  & \partial_2 g_{22}  & \partial_1 g_{22}
			\end{matrix} \right| }{4 \left| \begin{matrix}
				g_{11} & g_{12}  \\
				g_{21}  & g_{22}
			\end{matrix}\right|^2} \,\,,
		$$
		it can be easily concluded that
		\begin{enumerate}
			\item if some metric tensor element is zero (i.e. $g_{ij}(\theta) = 0 \,\, \forall \theta$), then $S(\theta) =0$ and then the geometry is flat. Hence all metric tensors of any one of the following forms
			$$
			\begin{pmatrix}
				0 & g_{12} \\
				g_{12} & g_{22}
			\end{pmatrix} \,\,,
			\begin{pmatrix}
				g_{11} & g_{12} \\
				g_{12} & 0
			\end{pmatrix}\,\,,
			\begin{pmatrix}
				g_{11} & 0 \\
				0 & g_{22}
			\end{pmatrix}
			$$
			induce a flat geometry.
			\item if two components are linearly dependent (i.e. $\exists \, \lambda \in \mathbb{R}| g_{ij} (\theta)  = \lambda g_{kl} (\theta)\,\, \forall \theta$), then $S(\theta) =0$ and the geometry is again flat. Hence all metric tensors of any one of the following forms
			$$
			\begin{pmatrix}
				g_{11} & \lambda g_{22} \\
				\lambda g_{22} & g_{22}
			\end{pmatrix}\,\,,
			\begin{pmatrix}
				g_{11} & \lambda g_{11} \\
				\lambda g_{11} & g_{22}
			\end{pmatrix}\,\,,
			\begin{pmatrix}
				\lambda g_{12} & g_{12} \\
				g_{12} & g_{22}
			\end{pmatrix}\,\,,
			\begin{pmatrix}
				g_{11} & g_{12} \\
				g_{12} & \lambda g_{12}
			\end{pmatrix}\,\,,
			\begin{pmatrix}
				\lambda g_{22} & g_{12} \\
				g_{12} & g_{22}
			\end{pmatrix}
			$$
			for some $\lambda \in \mathbb{R}$, induce a flat geometry.
			\item if all the metric tensor elements depend on just one of the two parameters, then $S(\theta) =0$ and the geometry is again flat. Hence all metric tensors of any one of the following forms
			$$
			\begin{pmatrix}
				g_{11} (\theta_1) & g_{12} (\theta_1) \\
				g_{12} (\theta_1) & g_{22} (\theta_1)
			\end{pmatrix},
			\begin{pmatrix}
				g_{11} (\theta_2) & g_{12} (\theta_2) \\
				g_{12} (\theta_2) & g_{22} (\theta_2)
			\end{pmatrix}
			$$
			induce a flat geometry.
		\end{enumerate}
	\end{theorem}
	\begin{proof}
		If a row or a column of the determinant at numerator of $S$
		$$
		\left|
		\begin{matrix}
			g_{11} & \partial_2 g_{11}  & \partial_1 g_{11}  \\
			g_{12}  & \partial_2 g_{12}  & \partial_1 g_{12}  \\
			g_{22}  & \partial_2 g_{22}  & \partial_1 g_{22}
		\end{matrix} \right|
		$$
		is equal to zero (or two rows or columns are linearly dependent), then the determinant itself is equal to zero, which implies $S(\theta) =0$.
	\end{proof}

	\subsection{Example of Exponential family}
	We will use an example from Nielsen \cite{Nielsen2013}: suppose $\theta= (\theta_1, \theta_2) = (\frac{\mu}{\sigma^2}, \frac{1}{2 \sigma^2})$. By defining
	$$
	k(x) = 0\,\,,
	$$
	$$
	h(x) = (x,-x^2)\,\,,
	$$
	$$
	\psi(\theta) = \frac{1}{4} \frac{\theta_1^2}{\theta_2} - \frac{1}{2} \log \theta_2 + \frac{1}{2} \log \pi\,\,,
	$$
	we have, after a few calculations, that
	$$
	p(x|\theta) = e^{\theta \cdot h(x) + k(x) -\psi(\theta)} = \frac{1}{\sqrt{2 \pi \sigma^2}} e^{-\frac{(x-\mu)^2}{2 \sigma^2}}
	$$
	i.e. we obtain the normal distribution.

	The metric tensor in this case is
	$$
	g_{ij} (\theta) =
	\begin{pmatrix}
		\frac{1}{2 \theta_2} & -\frac{\theta_1}{2 \theta_2^2} \\
		-\frac{\theta_1}{2 \theta_2^2} & \frac{\theta_1^2}{2 \theta_2^3} + \frac{1}{2 \theta_2^2}
	\end{pmatrix}\,\,.
	$$
	And it can be computed that the curvature is still
	$$
	S(\theta) = -\frac{1}{2}\,\,.
	$$

	We computed the curvature using a CAS by the formula involving determinants:
	$$
	S(\theta) = \frac{\left|
		\begin{matrix}
			g_{11}(\theta)  & \partial_2 g_{11}(\theta)  & \partial_1 g_{11}(\theta)  \\
			g_{12}(\theta)  & \partial_2 g_{12}(\theta)  & \partial_1 g_{12}(\theta)  \\
			g_{22}(\theta)  & \partial_2 g_{22}(\theta)  & \partial_1 g_{22} (\theta)
		\end{matrix} \right| }{4 \left| \begin{matrix}
			g_{11}(\theta)  & g_{12}(\theta)  \\
			g_{21}(\theta)  & g_{22}(\theta)
		\end{matrix}\right|^2}  = - \frac{1}{2}\,\,,
	$$
	obtaining the same result.

	\section{Beta manifold}

	The Beta distributions are of interest in applications because they may represent the probability to observe some probability or relative frequency (having domain in [0,1]): they have been repeatedly used for this purpose in applications \cite{Singh2014,Tiplica2020,Liu2017}. For this reason we will follow \cite{Brigant2019} and will study the asymptotic behavior of the manifold induced by the parametric family  of Beta distributions.
	\begin{definition}
		The Beta family is the parametric family of distributions
		$$
		p(x|\alpha, \beta) = \frac{\Gamma(\alpha+\beta)}{\Gamma(\alpha)\Gamma(\beta)} x^{\alpha-1}(1-x)^{\beta-1}, \forall x \in [0,1]
		$$
		where both parameters $\alpha$ and $\beta$ take values in $(0,+\infty)$ and where $\Gamma(x) = \int_0^{+\infty} {t^{x-1} e^{-t}} dt$ is Euler's Gamma function.
	\end{definition}

	Defining the digamma function
	$$ \psi (x) = \frac{d}{dx} \log \Gamma(x) $$
	and the trigamma function
	$$ \psi'(x) = \frac{d}{dx} \psi(x) = \frac{d^2}{dx^2} \log \Gamma (x) ,$$
	the corresponding metric tensor can be computed as
	$$
	g(\alpha,\beta) =
	\begin{pmatrix}
		\psi'(\alpha)-\psi'(\alpha+\beta) & -\psi'(\alpha+\beta) \\
		-\psi'(\alpha+\beta) & \psi'(\beta) - \psi'(\alpha+\beta)
	\end{pmatrix} \,\,.
	$$

	We can consequently compute the curvature, which is
	$$
	S(\alpha,\beta) =  \frac { \psi''(\alpha) \psi''(\beta) \psi '(\alpha+\beta) - \psi'(\alpha) \psi''(\beta) \psi''(\alpha+\beta) -\psi''(\alpha) \psi'(\beta) \psi''(\alpha+\beta) }{ 4 (\psi'(\alpha) \psi'(\alpha + \beta) + \psi'(\beta) \psi'(\alpha + \beta) -\psi'(\alpha) \psi'(\beta) )} \,\,.
	$$

	It can be proved \cite{Brigant2019} that the curvature is negative and bounded from above. However, we will not study the curvature over the whole manifold, limiting ourselves to reaching the same asymptotic result as in \cite{Brigant2019} with the aid of a CAS.
	\begin{theorem}
		Given the metric tensor
		$$
		g(\alpha,\beta) =
		\begin{pmatrix}
			\psi'(\alpha)-\psi'(\alpha+\beta) & -\psi'(\alpha+\beta) \\
			-\psi'(\alpha+\beta) & \psi'(\beta) - \psi'(\alpha+\beta)
		\end{pmatrix}
		$$
		of the Beta family of distributions, the asymptotic values for the curvature
		$S(\alpha,\beta)$ are
		$$
		\lim_{\alpha \rightarrow +\infty, \beta \rightarrow +\infty} {S(\alpha,\beta)} = - \frac{1}{2} \,\,,
		$$
		$$
		\lim_{\alpha \rightarrow 0, \beta \rightarrow 0} {S(\alpha,\beta)} = 0\,\,, and
		$$
		$$
		\lim_{\alpha \rightarrow +\infty, \beta \rightarrow 0}  {S(\alpha,\beta)} = \lim_{\alpha \rightarrow 0, \beta \rightarrow +\infty}  {S(\alpha,\beta)} = - \frac{1}{4}\,\,.
		$$
	\end{theorem}
	\begin{proof}
		Following \cite{Brigant2019} we will use two main approximations:
		$$
		\psi'(x) \sim_{x \rightarrow 0} \frac{1}{x^2}
		$$
		and
		$$
		\psi'(x) \sim_{x \rightarrow +\infty} \frac{1}{x}+\frac{1}{2 x^2} \,\,.
		$$\\\\
		Supposing $\alpha \rightarrow +\infty$ and $\beta \rightarrow +\infty$ then
		$$
		g(\alpha,\beta) =
		\begin{pmatrix}
			\psi'(\alpha)-\psi'(\alpha+\beta) & -\psi'(\alpha+\beta) \\
			-\psi'(\alpha+\beta) & \psi'(\beta) - \psi'(\alpha+\beta)
		\end{pmatrix} \sim
		$$
		$$
		\sim
		\begin{pmatrix}
			\frac{1}{\alpha} + \frac{1}{2 \alpha^2} - \frac{1}{\alpha+\beta} -\frac{1}{2 (\alpha+\beta)^2} & - \frac{1}{\alpha+\beta} -\frac{1}{2 (\alpha+\beta)^2}\\
			- \frac{1}{\alpha+\beta} -\frac{1}{2 (\alpha+\beta)^2} & \frac{1}{\beta} + \frac{1}{2 \beta^2} - \frac{1}{\alpha+\beta} -\frac{1}{2 (\alpha+\beta)^2}
		\end{pmatrix}\,\,.
		$$
		In this case we have
		$$
		S(\alpha,\beta) \sim \frac{-0.03125 \alpha^2 - 0.0625 \alpha \beta - 0.0625 \alpha - 0.03125 \beta^2 - 0.0625 \beta}{0.0625 \alpha^2 + 0.125 \alpha \beta + 0.125 \alpha + 0.0625 \beta^2 + 0.125 \beta + 0.0625} \sim -\frac{1}{2}
		$$
		as in \cite{Brigant2019}.\\\\

		Supposing $\alpha \rightarrow 0$ and $\beta \rightarrow 0$ then
		$$
		g(\alpha,\beta) =
		\begin{pmatrix}
			\psi'(\alpha)-\psi'(\alpha+\beta) & -\psi'(\alpha+\beta) \\
			-\psi'(\alpha+\beta) & \psi'(\beta) - \psi'(\alpha+\beta)
		\end{pmatrix}
		\sim
		$$
		$$
		\sim
		\begin{pmatrix}
			\frac{1}{\alpha^2} -\frac{1}{ (\alpha+\beta)^2} &  -\frac{1}{ (\alpha+\beta)^2}\\
			- \frac{1}{ (\alpha+\beta)^2} & \frac{1}{ \beta^2}  -\frac{1}{ (\alpha+\beta)^2}
		\end{pmatrix} \,\,.
		$$
		In this case the curvature
		$$
		S(\alpha,\beta) \sim 0
		$$
		as in \cite{Brigant2019}.\\\\

		Finally, supposing $\alpha \rightarrow 0$ and $\beta \rightarrow +\infty$ (or viceversa) then
		$$
		g(\alpha,\beta) =
		\begin{pmatrix}
			\psi'(\alpha)-\psi'(\alpha+\beta) & -\psi'(\alpha+\beta) \\
			-\psi'(\alpha+\beta) & \psi'(\beta) - \psi'(\alpha+\beta)
		\end{pmatrix}
		\sim
		$$
		$$
		\sim
		\begin{pmatrix}
			\frac{1}{ \alpha^2} - \frac{1}{\alpha+\beta} -\frac{1}{2 (\alpha+\beta)^2} & - \frac{1}{\alpha+\beta} -\frac{1}{2 (\alpha+\beta)^2}\\
			- \frac{1}{\alpha+\beta} -\frac{1}{2 (\alpha+\beta)^2} & \frac{1}{\beta} + \frac{1}{2 \beta^2} - \frac{1}{\alpha+\beta} -\frac{1}{2 (\alpha+\beta)^2}
		\end{pmatrix} \,\,.
		$$
		The result is a very complex polynomial, for which we can numerically compute
		$$
		S(\alpha,\beta) \sim -\frac{1}{4}
		$$
		again as in \cite{Brigant2019}.
	\end{proof}

	It can be thus observed that the Beta family of distributions has negative curvature (the manifold is locally hyperbolic). However, the absolute curvature of the Beta family is locally less than $\frac{1}{2}$ (the absolute curvature of the Gaussian family): while this could appear counterintuitive, it should be remarked that the Beta distribution does not have arbitrary mean and variance in $(-\infty,+\infty)$, since it can be easily shown that the mean of a Beta distribution lies inside the interval $[0,1]$ for every $\alpha,\beta > 0$.

	\section{Conclusions}
	We have shown that several of the most commonly used families of distributions with two parameters are locally hyperbolic, making the tools of hyperbolic geometry the standard in the study of probability manifold in most  application areas. In so doing, we have employed explainable computational techniques, exemplifying the fact that the use of such techniques makes it is possible to substantially boost scientific productivity.

    \bigskip

    \textit{Acknowledgments}: The research reported in this publication was supported by the Distinguished Professor Grant, \'Obuda University Budapest (Tender code: \'OE-KP-1-2022).\\\\

	%------
	% Insert acknowledgments and information
	% regarding funding at the end of the last
	% section, i.e., right before the bibliography.
	%------
	%
	%	\begin{ack}
		%		[We thank X. TODO]
		%	\end{ack}
	%
	%	\begin{funding}
		%		This work was partially supported by~\ldots [TODO]
		%	\end{funding}

	%------
	% Insert the bibliography.
	%------

	\bibliographystyle{plain}
	\bibliography{references}

\end{document}